\documentclass[10pt]{elsarticle}
\usepackage{graphicx}
\usepackage{subfigure}
\usepackage{enumerate}
\usepackage{amsmath,amsfonts}
\usepackage{amssymb}
\usepackage{array}
\usepackage{color}
\usepackage{arydshln}
\usepackage[english]{babel}
\usepackage{algorithm}
\usepackage{algpseudocode}
\newcommand{\ess}{\mathrm{ess}}

\newproof{proof}{Proof}

\newtheorem{theorem}{Theorem}[section]
\newtheorem{lemma}[theorem]{Lemma}
\newtheorem{corollary}[theorem]{Corollary}
\newtheorem{definition}[theorem]{Definition}
\newtheorem{example}[theorem]{Example}

\begin{document}

\begin{frontmatter}

\title{On the Phases of a Semi-Sectorial Matrix}
\author[label2]{Li Qiu\corref{cor1}}
\ead{eeqiu@ust.hk}
\cortext[cor1]{Corresponding author}
\author[label2]{Dan Wang}
\ead{dwangah@connect.ust.hk}
\author[label2]{Xin Mao}
\ead{xmaoaa@connect.ust.hk}
\author[label3]{Wei Chen}
\ead{w.chen@pku.edu.cn}

\address[label2]{Department of Electronic and Computer Engineering, Hong Kong University of Science and Technology, Clear Water Bay, Kowloon, Hong Kong, China}
\address[label3]{Department of Mechanics and Engineering Science \& State Key Laboratory for Turbulence and Complex Systems, Peking University, Beijing 100871, China}

\begin{abstract}
In this paper, we extend the definition of phases of sectorial matrices to those of semi-sectorial matrices, which are possibly singular. Properties of the phases are also extended, including those of the Moore-Penrose generalized inverse, compressions and Schur complements, matrix sums and products. In particular, a majorization relation is established between the phases of the nonzero eigenvalues of $AB$ and the phases of the compressions of $A$ and $B$, which leads to a generalized matrix small phase theorem. For the matrices which are not necessarily semi-sectorial, we define their (largest and smallest) essential phases via diagonal similarity transformation. An explicit expression for the essential phases of a Laplacian matrix of a directed graph is obtained.
\end{abstract}


\end{frontmatter}

\section{Introduction}

Recently, we studied the phases of a class of complex matrices called sectorial matrices \cite{WCKQ2020}.
Here we will extend the study to those of a wider class of matrices called semi-sectorial matrices.
A challenge in this new study is that in general an $n \times n$ semi-sectorial matrix does not have $n$ phases.
Generalizing the results in \cite{WCKQ2020} requires great attention in keeping track of the numbers of phases for the matrices involved.

\section{Matrix Phases}

The numerical range, also called field of
values, of a matrix $C \!\in\! \mathbb{C}^{n\times n}$ is defined as
\[
W(C) = \{ x^*Cx: x \in \mathbb{C}^n \mbox{ with } \|x\|=1\},
\]
which, as a subset of
$\mathbb{C}$, is compact and convex, and contains the spectrum of $C$ \cite{horntopics}. Furthermore, the angular numerical range, also called angular field of values, of $C$ is defined as
\[
W'(C) = \{ x^*Cx: x \in \mathbb{C}^n, x\neq 0 \},
\]
which is the conic hull of $W(C)$ and is always a convex cone. The field angle of $C$, denoted by $\delta(C)$, is defined as the angle subtended by $W'(C)$ if $W'(C)$ is salient, i.e., does not contain a line through the origin, as $\pi$ if $W'(C)$ contains one line through the origin, and as $2\pi$ if $W'(C)$ is the whole complex plane. See \cite{horntopics} for more details.

\begin{definition}
\    \
\begin{enumerate}
\item $C$ is said to be sectorial if $0\notin W(C)$.
\item $C$ is said to be quasi-sectorial if $\delta(C) < \pi$.
\item $C$ is said to be semi-sectorial if $\delta(C) \leq \pi$.
\end{enumerate}
\end{definition}

A sectorial matrix $C$ has its numerical range $W(C)$ contained in an open half complex plane due to its convexity and hence $\delta (C) < \pi$. It is known
that a sectorial $C$ is congruent to a diagonal unitary matrix that is unique up to a permutation \cite{Horn,ZhangFuzhen2015}, i.e., there exist a nonsingular matrix
$T$ and a diagonal unitary matrix $D$ such that
$C=T^*DT$.
This factorization is called sectorial decomposition in \cite{ZhangFuzhen2015}. 
In such a factorization, the eigenvalues (i.e., the diagonal elements) of $D$ are distributed in an arc on the unit circle with length less than $\pi$.
We can then attempt to define the phases of $C$, denoted by $$\overline{\phi}(C)=\phi_1(C)\geq\phi_2(C)\geq\dots\geq\phi_n(C)=\underline{\phi}(C),$$ to be the phases of the eigenvalues of $D$ so that $\overline{\phi}(C)-\underline{\phi}(C) <\pi$. The phases of $C$ defined in this way are not uniquely determined,
but are rather determined modulo $2\pi$. If we make a selection of $\displaystyle \gamma(C) = [\overline{\phi}(C)+\underline{\phi}(C)]/2$, called the phase center of $C$, in $\mathbb{R}$, then the phases are uniquely determined. The phases are said to take the principal values if $\gamma (C)$ is selected in $(-\pi, \pi]$. The phases defined in this fashion resemble the canonical angles of $C$ introduced in \cite{FurtadoJohnson2001}. Let us denote
\[
\phi (C) = \begin{bmatrix} \phi_1 (C) & \phi_2 (C) & \cdots & \phi_n(C) \end{bmatrix}.
\]

A graphic interpretation of the phases is illustrated in \mbox{Fig. \ref{fig1}}. The two angles from the positive real axis to each of the two supporting rays of $W(C)$ are $\overline{\phi}(C)$ and $\underline{\phi}(C)$ respectively. The other phases of $C$ lie in between.
\begin{figure}[htb]
\centering
\includegraphics[scale=0.5]{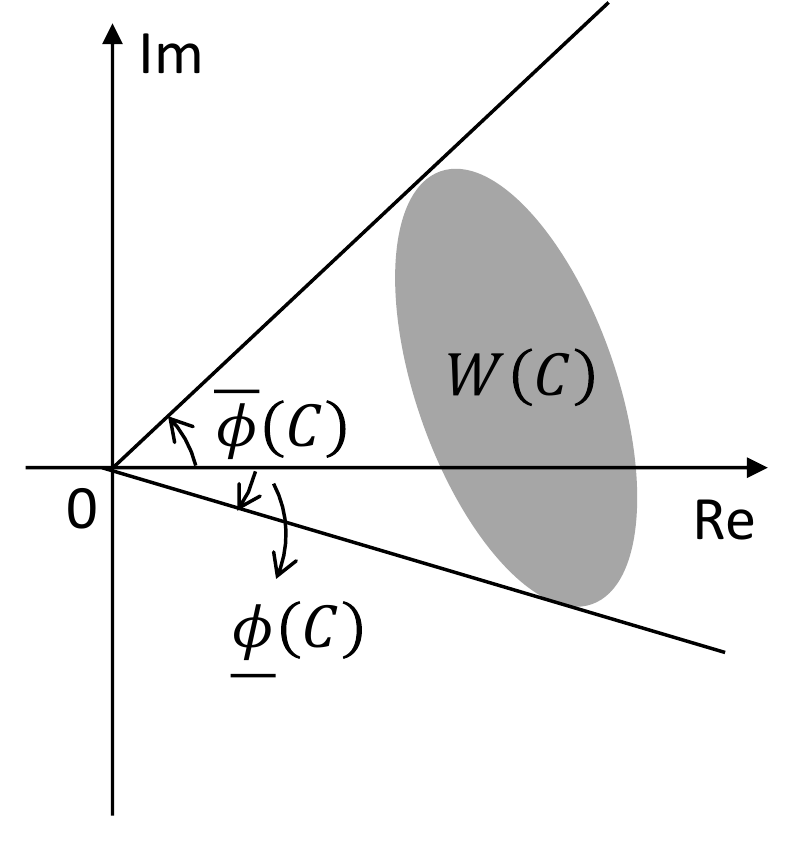}
\vspace{-10pt}
\caption{Geometric interpretation of $\overline{\phi}(C)$ and $\underline{\phi}(C)$.}
\label{fig1}
\end{figure}

An example of the numerical range of a quasi-sectorial matrix is shown in Fig.~\ref{fig:semi-sectorial}(a).
We see that 0 is a sharp point of the boundary of the numerical range. Let $r=\mathrm{rank} (C)$. Then a quasi-sectorial $C$ has a decomposition
\begin{align}
C= U \begin{bmatrix} 0 & 0 \\ 0 & C_s \end{bmatrix} U^*\label{quasis}
\end{align}
where $U$ is unitary and $C_s \in \mathbb{C}^{r \times r}$ is sectorial \cite{FurtadoJohnson2003}, i.e., the range and kernel of $C$ are orthogonal and the compression of $C$ to its range is sectorial. The phases of $C$ are then defined as the phases of $C_s$.
Hence an $n \times n$ rank $r$ quasi-sectorial matrix $C$ has $r$ phases satisfying
\[
\overline{\phi}(C)=\phi_1(C)\geq\phi_2(C)\geq\dots\geq\phi_r(C)=\underline{\phi}(C).
\]

\begin{figure}[htb]
\centering
\subfigure[Numerical range of a quasi-sectorial matrix.]{
\includegraphics[scale=0.44]{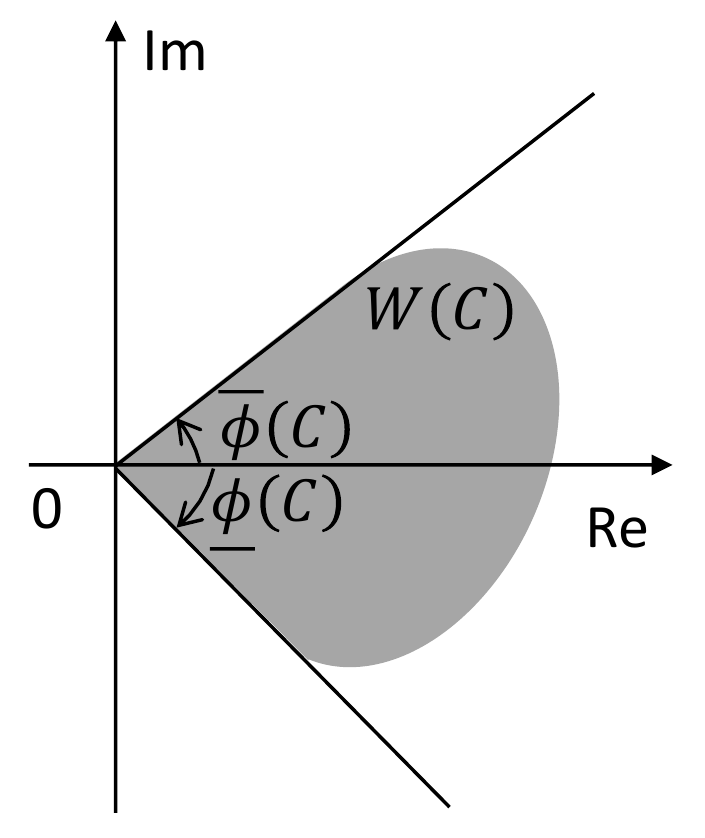}}
\hspace{30pt}
\subfigure[Numerical range of a semi-sectorial matrix.]{
\includegraphics[scale=0.44]{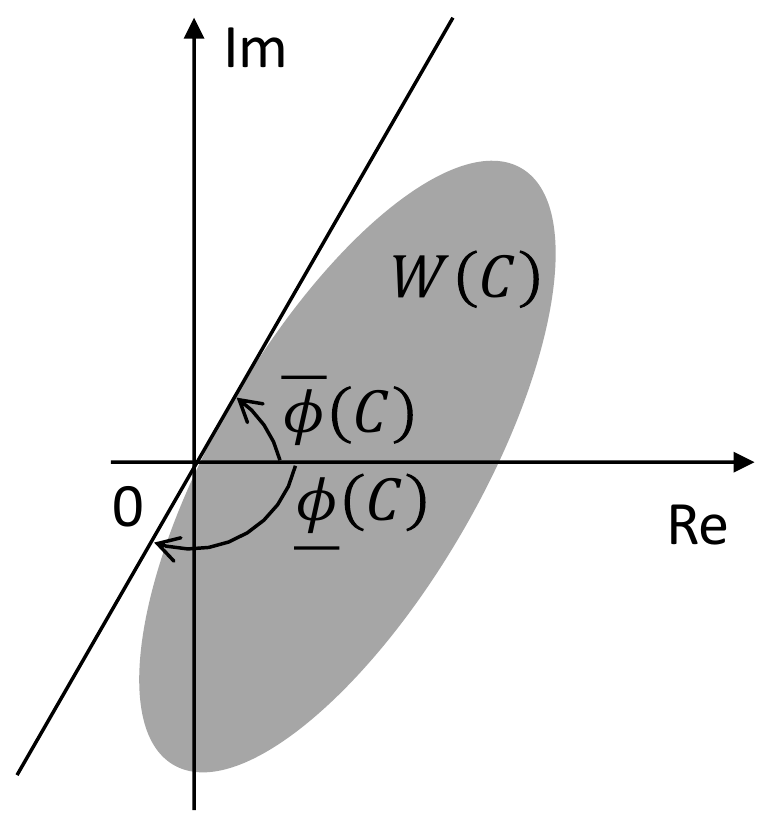}}
\caption{Origin is on the boundary of the numerical range.}
\label{fig:semi-sectorial}
\end{figure}

Here comes a question. While an $n \times n$ zero matrix is clearly not sectorial, is it quasi-sectorial?
The answer is affirmative. It has zero number of phases and following conventions we have $\overline{\phi} (0)=-\infty$ and $\underline{\phi}(0)=\infty$.

A typical example of the numerical range of a  semi-sectorial matrix is illustrated in Fig.~\ref{fig:semi-sectorial}(b). A degenerate case of semi-sectorial matrices is when the numerical range has no interior and is given by a straight interval containing origin in its relative interior.
In this case, $C$ is said to be rotated Hermitian and is subject to the following decomposition
\[
C= T^* \mathrm{diag} \{ 0_{n-r}, e^{j(\theta_0+\pi/2)} I,  e^{j(\theta_0-\pi/2)} I \} T.
\]
Here $\theta_0$, equal to the phase center $\gamma(C)$, is determined modulo $\pi$. It has two possible principal values in $(-\pi, \pi]$. The phases of $C$ are several copies of $\theta_0+\pi/2$
and several copies of $\theta_0-\pi/2$.

The generic case of semi-sectorial matrices is when the numerical range has nonempty interior. It is known \cite{FurtadoJohnson2003} that such a matrix has the following generalized sectorial decomposition
\begin{equation}
C=T^*\begin{bmatrix} 0_{n-r} & 0 & 0 \\ 0 & D & 0 \\  0 & 0 & E \end{bmatrix}T,
\label{gsd}
\end{equation}
where
\[
D=\mathrm{diag} \{ e^{j\theta_1}, \ldots , e^{j\theta_m}\}
\]
with $\theta_0+\pi/2 \geq \theta_1  \geq \cdots \geq \theta_m \geq \theta_0-\pi/2$ and
\[
E=\mathrm{diag} \left\{ e^{j\theta_0} \begin{bmatrix} 1 & 2 \\ 0 & 1 \end{bmatrix}, \dots , e^{j\theta_0} \begin{bmatrix} 1 & 2 \\ 0 & 1 \end{bmatrix} \right\}.
\]
In this case, the phases of $C$ are defined as $\theta_1, \ldots \theta_m$ together with $(r-m)/2$ copies of
$\theta_0\pm\pi/2$.

The notion of matrix phases subsumes the well-studied accretive and strictly accretive matrices \cite{Kato}, i.e., matrices with positive semi-definite and positive definite Hermitian parts respectively. In particular, the phases of a sectorial $C$ take principal values in $(- \pi/2, \pi/2)$ if and only if $C$ is strictly accretive; the phases of a semi-sectorial $C$ take principal values  in $[-\pi/2, \pi/2]$ if and only if $C$ is accretive. What is the role of quasi-sectorial matrices? A quasi-sectorial accretive matrix is called a quasi-strictly accretive matrix. A quasi-strictly accretive matrix cannot be identified from its Hermitian part. For example, $\begin{bmatrix} 1 & 0 \\ 0 & 0 \end{bmatrix}$ is quasi-strictly accretive while $\begin{bmatrix} 1 & 1 \\ -1 & 0 \end{bmatrix}$ is not though they have the same Hermitian part.

\section{Computation of Phases}

Computation involving numerical range is usually considered as complicated. However, the computation of phases can be rather easy for quasi-sectorial matrices. First, if $C$ is known to be quasi-strictly accretive, let
\[
C=T^*\begin{bmatrix} 0&0 \\ 0&D\end{bmatrix}T
\]
be a generalized sectorial decomposition. Then
\[
C(C^*)^{\dagger}= T^*\begin{bmatrix}0&0\\0&D^2\end{bmatrix}(T^*)^{-1}.
\]
This means that the phases of $C$, as the phases of eigenvalues of $D$, are the halves of the phases of the nonzero eigenvalues of $C(C^*)^{\dagger}$, taking principal values in $(-\pi, \pi)$.

More generally, we have the following characterization of quasi-sectorial matrices.

\begin{lemma}\label{quasilemma}
A matrix $C$ is quasi-sectorial with phases taking principal values in $(-\frac{\pi}{2}+\alpha,\frac{\pi}{2}+\alpha)$, where $\alpha\in(-\pi,\pi]$, if and only if there exists $\epsilon>0$ such that
\begin{align}
e^{-j\alpha}C+e^{j\alpha}C^*\geq \epsilon C^*C.\label{quasiine}
\end{align}
\end{lemma}

\begin{proof}
We start with the necessity. Since $C$ is quasi-sectorial, it has a decomposition (\ref{quasis}), where $U$ is unitary and $C_s$ is sectorial with phases taking principal values in $(-\frac{\pi}{2}+\alpha,\frac{\pi}{2}+\alpha)$. It follows that $e^{-j\alpha}C_s+e^{j\alpha}C_s^*>0$ and thus there exists a small $\epsilon>0$ such that
\[
e^{-j\alpha}C_s+e^{j\alpha}C_s^*\geq \epsilon C_s^*C_s,
\]
which implies \eqref{quasiine}.

Next, we show the sufficiency. Suppose there exists $\epsilon>0$ such that the inequality (\ref{quasiine}) holds. Then, $e^{-j\alpha}C+e^{j\alpha}C^*\geq 0$, meaning that $C$ is semi-sectorial. Thus, $C$ has a decomposition $C=U\begin{bmatrix}0&0\\0&\tilde{C}\end{bmatrix}U^*$, where $U$ is unitary and $\tilde{C}$ is nonsingular. Then,
\begin{align*}
    e^{-j\alpha}C+e^{j\alpha}C^*&=U\begin{bmatrix}0&0\\0&e^{-j\alpha}\tilde{C}+e^{j\alpha}\tilde{C}^*\end{bmatrix}U^*\geq \epsilon C^*C=U\begin{bmatrix}0&0\\0&\epsilon \tilde{C}^*\tilde{C}\end{bmatrix}U^*.
\end{align*}
This implies that $e^{-j\alpha}\tilde{C}+e^{j\alpha}\tilde{C}^*\geq \epsilon \tilde{C}^*\tilde{C}>0$ and thus $\tilde{C}$ is sectorial with phases taking principal values in $(-\frac{\pi}{2}+
\alpha, \frac{\pi}{2}+\alpha)$.
\end{proof}

Lemma 3.1 says that a matrix is quasi-sectorial if and only if it can be rotated to a quasi-strictly accretive matrix by multiplying $e^{-j\alpha}$. If we can find this rotation, then we can compute the phases of the resulted quasi-strictly accretive matrix using the method above. Adding $\alpha$ to these phases then results the phases of the given quasi-sectorial matrix. Finding the right rotation now is the key issue. Instead of multiplying $e^{-j\alpha}$ as in Lemma \ref{quasilemma}, let us do the rotation by multiplying a nonzero complex number $z=x+jy$ to $C=A+jB$, where $A=\frac{C+C^*}{2}$ and $B=\frac{C-C^*}{2j}$, so to convert condition \eqref{quasiine} to a linear matrix inequality (LMI):
\[
zC + z^*C^* = 2 (xA-yB) \geq \epsilon C^*C
\]
In such a case, the phases of $C$ are the phases of $zC$ minus $\angle z$. Hence the computation of phases boils down to solving an LMI followed by a matrix eigenvalue computation.

The phase computation for a semi-sectorial matrix might be numerically problematic and we leave it for future research.

\section{Simple Properties}

Some properties of the phases of semi-sectorial matrices can be obtained easily from those of sectorial matrices.

First, the phases of the Moore-Penrose generalized inverse $C^\dagger$ can be obtained from those of $C$ easily.

\begin{lemma}\label{lem:inverse}
Let $C\in \mathbb{C}^{n \times n}$ be semi-sectorial with rank $r$. Then
\begin{equation}\label{eq:inverse}
\phi_i(C^\dagger)=-\phi_{r-i+1} (C).
\end{equation}
\end{lemma}

Next, the set of phase bounded semi-sectorial matrices
\begin{equation*}
\mathcal{C}[\alpha, \beta] =
\left\{C\in \mathbb{C}^{n\times n}: C \text{ is semi-sectorial and } [\underline{\phi}(C), \overline{\phi}(C)] \subset [\alpha, \beta] \right\}
\end{equation*}
where $0\leq \beta-\alpha<2\pi$, is a cone. The following lemma, which can be proved using the same way as for the sectorial case \cite{WCKQ2020}, implies that if $\beta-\alpha\leq \pi$, then $\mathcal{C}[\alpha, \beta]$ is a closed convex cone.

\begin{lemma}
\label{convexcone}
If $\beta-\alpha\leq \pi$, $A, B \in \mathcal{C}[\alpha, \beta]$, then $A+B \in \mathcal{C}[\alpha, \beta]$.
\end{lemma}

Another property concerns the phases of the compression of a matrix. For $C\in\mathbb{C}^{n\times n}$, the matrix $\tilde{C}\!=\!X^*CX$, where $X\!\in\!\mathbb{C}^{n\times k}$ has full column rank, is said to be a compression of $C$. The phases of $\tilde{C}$ and those of $C$ have the following interlacing property, which is a simple extension of \cite[Lemma 7]{FurtadoJohnson2003}.
\begin{lemma}
\label{lemma:semi-compression}
Let $C\in\mathbb{C}^{n\times n}$ be a nonzero semi-sectorial matrix and $\tilde{C}$ be a nonzero compression of $C$. Denote $r=\mathrm{rank} (C)$ and $\tilde{r}=\mathrm{rank} (\tilde{C})$. Then $\tilde{C}$ is semi-sectorial and
\begin{align}\label{eq:interlacing}
\phi_j(C) \geq \phi_j(\tilde{C}) \geq \phi_{r-\tilde{r}+j}(C),\ j=1,\dots, \tilde{r}.
\end{align}
\end{lemma}

By exploiting the properties of the phases of compressions of a semi-sectorial matrix, we can derive the phases of the generalized Schur complement of a semi-sectorial matrix. Let $C\in\mathbb{C}^{n\times n}$ be partitioned as
$C=\begin{bmatrix}C_{11}&C_{12}\\C_{21}&C_{22}\end{bmatrix}$, where $C_{11}\in\mathbb{C}^{k\times k}$. The generalized Schur complement \cite{ZhangFuzhen2006Schur} of $C_{11}$ in $C$, denoted by $C\slash_{11}$, is defined as $C\slash_{11}=C_{22}-C_{21}C_{11}^{\dagger}C_{12}$. Denoted by $\mathcal{R}(C)$ the range of matrix $C$.
\begin{lemma}
\label{lem:schur}
Let $C=\begin{bmatrix}C_{11}&C_{12}\\C_{21}&C_{22}\end{bmatrix}$ be semi-sectorial and $C\slash_{11}$ be the generalized Schur complement of $C_{11}$ in $C$. Denote $r=\mathrm{rank} (C)$ and $\hat{r}=\mathrm{rank} (C\slash_{11})$. If $\mathcal{R}(C_{12})\subset\mathcal{R}(C_{11})$ and $\mathcal{R}(C_{21}^*)\subset\mathcal{R}(C_{11}^*)$, then $C\slash_{11}$ is semi-sectorial and
\begin{equation}
\label{eq:schur}
\phi_j(C) \geq \phi_j(C\slash_{11}) \geq \phi_{r-\hat{r}+j}(C),\ j=1,\dots, \hat{r}.
\end{equation}
\end{lemma}

\begin{proof}
Since $\mathcal{R}(C_{12})\subset \mathcal{R}(C_{11})$, $\mathcal{R}(C_{21}^*)\subset\mathcal{R}(C_{11}^*)$, it follows from \cite{burns1974} that $C^{\dagger}=\begin{bmatrix} ? & ? \\ ? & (C\slash_{11})^\dagger
\end{bmatrix}$. By Lemmas \ref{lem:inverse} and \ref{lemma:semi-compression},  $(C\slash_{11})^{\dagger}$ is semi-sectorial and so is $C\slash_{11}$. In view of \eqref{eq:interlacing}, we have
\[
\phi_j(C^{\dagger}) \geq \phi_j((C\slash_{11})^\dagger) \geq \phi_{\mathrm{rank} (C^\dagger)-\mathrm{rank}((C\slash_{11})^\dagger)+j} (C^{\dagger}),
\]
for $j=1,\dots, \mathrm{rank}((C\slash_{11})^\dagger).$
Obviously $\mathrm{rank}(C^{\dagger})=r$, $\mathrm{rank}((C\slash_{11})^\dagger)=\hat{r}$. In view of \eqref{eq:inverse}, we obtain \eqref{eq:schur}.
\end{proof}

Note that if $C$ is a quasi-sectorial matrix in Lemma \ref{lem:schur}, it can be inferred from \cite[Lemma 2.4]{mitra1983} that there naturally hold $\mathcal{R}(C_{12})\subset \mathcal{R}(C_{11})$ and $\mathcal{R}(C_{21}^*)\subset\mathcal{R}(C_{11}^*)$. This means that for quasi-sectorial $C$, the generalized Schur complement $C\slash_{11}$ is quasi-sectorial and \eqref{eq:schur} holds.

\section{Matrix Product}

Given two vectors $x,y\in \mathbb{R}^n$, denote by $x^\downarrow$ and $y^\downarrow$ the rearranged
versions of $x$ and $y$, respectively, in which their elements are sorted in a non-increasing order. Then, $x$ is said to be majorized by $y$, denoted by $x\prec y$, if
\begin{align*}
\sum_{i=1}^k x^\downarrow_i\leq\sum_{i=1}^k y^\downarrow_i,\ k=1,\dots, n-1,\quad \text{and} \quad
\sum_{i=1}^n x^\downarrow_i=\sum_{i=1}^n y^\downarrow_i.
\end{align*}

If $A, B \in \mathbb{C}^{n\times n}$ are sectorial matrices, it is known from \cite{WCKQ2020} that
\[
\angle\lambda(C) \prec \phi(A)+\phi(B).
\]
where $\lambda(C)$ is the vector of the $n$ eigenvalues of $C$ and $\angle$ is applied element-wise. In this section, we will extend this majorization relation to the case
when $A$ and $B$ are semi-sectorial. The difficulty is that now the matrices $A$ and $B$ may have different numbers of phases and $AB$ may have another different number of nonzero eigenvalues. Let us denote the vector of nonzero eigenvalues of matrix $C$ as $\lambda_{\neq 0}(C)$.

\begin{example}
Let
\[
A=\begin{bmatrix} 0 & 0 \\
0 & 1
\end{bmatrix}, \ \ \ B=\begin{bmatrix} 0 & -1 \\
1 & 0
\end{bmatrix}.
\]
Then $A$ has 1 phase, $B$ has 2 phases, and $AB$ has 0 nonzero eigenvalues.
Also notice that $\mathrm{rank} (AB) = 1$ while $\mathrm{rank} (AB)^2 = 0$.
\end{example}



\begin{theorem}
\label{thm:semi-major}
Let $A, B\in\mathbb{C}^{n\times n}$ be quasi-sectorial and semi-sectorial with phase centers $\gamma(A)$ and $\gamma(B)$ respectively. Let $C=AB$. Then
\begin{enumerate}
\item $C$ has $r\!=\!\mathrm{rank} (C^2)$ nonzero eigenvalues $\lambda_{\neq 0}(C) \!\!=\!\! \begin{bmatrix}\lambda_1(C) \!\!&\!\cdots\!&\!\! \lambda_r(C)\end{bmatrix}$,

\item $\angle \lambda_i (C)$ can take values so that
\begin{equation*}
\gamma(A)+\gamma(B)-\pi
< \angle \lambda_r (C) \leq \dots \leq \angle \lambda_1 (C)
< \gamma(A)+\gamma(B)+\pi,
\end{equation*}
\item there exists an $n \times r$ isometry $U$ such that $U^*AU$ is sectorial, $U^*BU$ is nonsingular semi-sectorial and
$\angle\lambda_{\neq 0} (C) \prec \phi(U^*AU)+\phi(U^*BU)$.
\end{enumerate}
\end{theorem}

\begin{proof}
It is well known $C$ has $\mathrm{rank} (C^k)$ nonzero eigenvalues for each $k\geq \mathrm{Ind} (C)$ which is the index of $C$. In our particular case, we will show that $\mathrm{Ind} (C) \leq 2$.
Let $U_1$ be an isometry onto $\mathcal {R} (A)$ and $U_1^\perp$ be an isometry onto $\mathcal{N}(A)$. Then
\begin{align*}
\begin{bmatrix} U_1^\perp & U_1 \end{bmatrix}^* A \begin{bmatrix} U_1^\perp & U_1 \end{bmatrix}
& =\begin{bmatrix} 0& 0 \\ 0 & A_1 \end{bmatrix}, \\
\begin{bmatrix} U_1^\perp & U_1 \end{bmatrix}^* B \begin{bmatrix} U_1^\perp & U_1 \end{bmatrix}
& =\begin{bmatrix} ? & ? \\ X & B_1 \end{bmatrix} .
\end{align*}
The nonzero eigenvalues of $C$ are the same as the nonzero eigenvalues of $A_1B_1$. Since $A$ is quasi-sectorial and $B$ is semi-sectorial, it follows that $A_1$ is sectorial and $B_1$ is semi-sectorial.
Let $U_2$ be an isometry onto $\mathcal {R} (B_1)$ and $U_2^\perp$ be an isometry onto $\mathcal{N}(B_1)$. Then
\begin{align*}
\begin{bmatrix} U_2^\perp & U_2 \end{bmatrix}^* A_1 \begin{bmatrix} U_2^\perp & U_2 \end{bmatrix}
& =\begin{bmatrix} ? & Y \\ ? & A_2 \end{bmatrix}, \\
\begin{bmatrix} U_2^\perp & U_2 \end{bmatrix}^* B_1 \begin{bmatrix} U_2^\perp & U_2 \end{bmatrix}
& =\begin{bmatrix} 0 & 0 \\ 0 & B_2 \end{bmatrix}.
\end{align*}
Since $A_1$ is sectorial and $B_1$ is semi-sectorial, it follows that $A_2$ is sectorial and $B_2$ is semi-sectorial. By the non-singularity of $A_2$ and $B_2$, we see that the nonzero eigenvalues of $A_1B_1$, and hence those of $AB$, are exactly the eigenvalues of $A_2B_2$. Note that $C$ is unitarily similar to
\[
\begin{bmatrix}
0 & 0 & 0 \\ Z_1 & 0 & YB_2 \\ Z_2 & 0 & A_2B_2
\end{bmatrix},
\]
where $\begin{bmatrix} {Z_1} \\ {Z_2}\end{bmatrix}  \!=\! A_1X$. Hence $C^2$ is unitarily similar to
\[
\begin{bmatrix} 0 & 0 & 0 \\ YB_2Z_2 & 0 & YB_2 A_2B_2\\ A_2B_2 Z_2 & 0 & (A_2B_2)^2 \end{bmatrix}.
\]
Since $A_2B_2$ is nonsingular, we have
\[
\mathrm{rank} (C^2) = \mathrm{rank} \begin{bmatrix} YB_2Z_2 & YB_2 A_2B_2\\ A_2B_2 Z_2 & (A_2B_2)^2 \end{bmatrix} = \mathrm{rank} (A_2B_2).
\]
This shows item 1.

Let $U=\mathcal{R}(U_1U_2)$. Then $A_2B_2= (U^*AU)(U^* BU)$. Here $U^*AU$ is sectorial and $U^*BU$ is semi-sectorial but nonsingular. Since they are compressions, it follows from Lemma \ref{lemma:semi-compression} that
$[\underline{\phi}(U^*AU), \overline{\phi}(U^*AU)] \subset [\underline{\phi}(A), \overline{\phi}(A)]$ and
$[\underline{\phi}(U^*BU), \overline{\phi}(U^*BU)] \subset [\underline{\phi}(B), \overline{\phi}(B)]$.

Let $B_\epsilon =U^*BU + \epsilon e^{j\gamma(U^*BU)}I$ for some $\epsilon >0$, then $B_\epsilon$ is sectorial and $\lim_{\epsilon \downarrow 0} B_\epsilon= U^*BU$. By continuity, $\lim_{\epsilon \downarrow 0} \phi(B_\epsilon) = \phi(U^*BU)$.
It follows from Theorem 6.2 in \cite{WCKQ2020} that $\lambda (U^*AU B_\epsilon)$
can take values in
\begin{equation*}
[\underline{\phi}(U^*AU)+\underline{\phi}(U^*BU), \overline{\phi}(U^*AU)+\overline{\phi}(B_\epsilon)]
\subset (\gamma(A)+\gamma(B_\epsilon)-\pi, \gamma(A)+\gamma(B_\epsilon)+\pi)
\end{equation*}
and
\[
\angle \lambda (U^*AU B_\epsilon) \prec \phi(U^*AU) + \phi(B_\epsilon).
\]
Taking limits in both sides, we get
\[
\angle \lambda (U^*AU U^*BU) \prec \phi(U^*AU) + \phi(U^*BU)
\]
that shows items 2 and 3.
\end{proof}

It is worth noting that item 3 in Theorem 5.1 also implies
\[
\underline{\phi}(A)+ \underline{\phi} (B) \leq \angle \lambda_i(C) \leq \overline{\phi}(A)+ \overline{\phi}(B) .
\]

By taking $A=I$ and $B=C$ in the above theorem, we obtain a majorization relation between the phases of nonzero eigenvalues and the matrix phases of a semi-sectorial matrices.

\begin{corollary}
Let $C$ be semi-sectorial. Then $C$ has $r=\mathrm{rank} (C)$ nonzero eigenvalues and there exists an $n \times r$ isometry $U$ such that $U^*CU$ is nonsingular semi-sectorial and
\[
\angle \lambda_{\neq 0} (C) \prec \phi(U^*CU).
\]
\end{corollary}

\begin{proof}
The only deviation of this corollary from Theorem 5.2 is that the number of nonzero eigenvalues of $C$ is equal to $\mathrm{rank} (C)$, instead of
$\mathrm{rank} (C^2)$, which follows immediately from decomposition (2).
\end{proof}

\section{Matrix Small Phase Theorem}

The singularity of matrix $I + AB$ plays an important role in the stability analysis of feedback systems \cite{CWKQcdc19,CWKQ2021,MCQ2022}. In \cite{WCKQ2020}, we have shown that for a sectorial matrix $A\in\mathbb{C}^{n\times n}$ with phases in $(-\pi, \pi)$ and $\alpha\in[0,\pi)$, there holds that $\mathrm{rank}(I+AB) =n$ for all $B\in \mathcal{C}[-\alpha, \alpha]$ if and only if $\alpha<\min\{\pi-\overline{\phi}(A), \pi+\underline{\phi}(A)\}$. Now we generalize it in two fronts: 1) semi-sectorial matrices are considered; 2) the phase sectors which $A$ and $B$ respectively belong to are more general.

\begin{theorem}[Matrix small phase theorem] \label{matrix-spt}
Let $A \in \mathbb{C}^{n \times n}$ be a quasi-sectorial matrix with $\gamma (A) \in \mathbb{R}$. Then $\mathrm{det}\,(I + AB)\neq 0$ for all $B\in \mathcal{C}[\alpha, \beta]$ if and only if $[\alpha, \beta] \subset  (-\pi-\underline{\phi}(A), \pi- \overline{\phi}(A))$ modulo $2\pi$.
\end{theorem}

\begin{proof} The sufficiency follows from Theorem~\ref{thm:semi-major} easily. It remains to show the necessity. Since $A$ is quasi-sectorial, there is a nonsingular $T$ such that
$A=T^* \begin{bmatrix} 0 & 0 \\ 0 & D \end{bmatrix} T$,
where $D$ is a diagonal unitary matrix. If $[\alpha, \beta] \not\subset (-\pi-\underline{\phi}(A), \pi-\overline{\phi}(A))$, then either $\overline{\phi}(A)+\beta\geq \pi$ or $\underline{\phi}(A)+\alpha\leq -\pi$. Consider the case when $\overline{\phi}(A)+\beta\geq \pi$. Let $\lambda_1=e^{j(\pi-\beta)}$ and $\lambda=\begin{bmatrix}\lambda_1&\lambda_2&\cdots&\lambda_r\end{bmatrix}'$, where $\angle \lambda_1\geq \angle \lambda_2\geq \dots\geq \angle \lambda_r$, be such that
$\angle\lambda \prec \phi(D)$. This can always be done as $\angle \lambda_1=\pi-\beta\leq \overline{\phi}(D)$. By \cite[Theorem 1]{Horn}, there exists a nonsingular $M\in\mathbb{C}^{r\times r}$ with polar decomposition $M\!=\!PU$ such that $\lambda(M)\!=\!\lambda$ and $\angle \lambda(U)\!=\!\phi(D)$. Then $\tilde{M}\!=\!P^{-\frac{1}{2}}MP^{\frac{1}{2}}=P^{\frac{1}{2}}UP^{\frac{1}{2}}$ is sectorial with $ \lambda(\tilde{M})\!=\!\lambda$ and $\phi(\tilde{M})\!=\!\phi(D)$. Hence $\tilde{M}\!=\!\tilde{T}^*D\tilde{T}$ for some nonsingular $\tilde{T}$. Now, let
\[
B=T^{-1}\begin{bmatrix}0&0\\0&e^{j\beta}\tilde{T}\tilde{T}^{*}\end{bmatrix}T^{-*}.
\]
Clearly, $B\in\mathcal{C}[\alpha,\beta]$. In addition,
\[
I+AB=I+T^*\begin{bmatrix}0&0\\0&\tilde{T}^{-*}e^{j\beta}\tilde{M}\tilde{T}^{*}\end{bmatrix}T^{-*}
\]
loses rank. The case when $\underline{\phi}(A)+\alpha\leq -\pi$ is similar.
\end{proof}

\section{Essential Phases of a Matrix}
In many applications, we may encounter a matrix which is not necessarily semi-sectorial but can be made semi-sectorial by diagonal similarity transformation. Such a matrix is said to be essentially semi-sectorial. For such a matrix $C$, we define its (largest and smallest) essential phases to be
\begin{align*}
\overline\phi_{\mathrm{ess}}(C)\!=\!\inf_{D\in\mathcal{D}} \overline\phi(D^{-1}CD),\ \
\underline\phi_{\mathrm{ess}}(C)\!=\!\sup_{D\in\mathcal{D}} \underline \phi(D^{-1}CD),
\end{align*}
where $\mathcal{D}$ is the set of positive definite diagonal matrices. Here the infimum and supremum are taken over $D\in\mathcal{D}$
such that $D^{-1}CD$ is semi-sectorial and $\overline\phi(D^{-1}CD)$ and $\underline \phi(D^{-1}CD)$ take their principal values. Such an essential phase problem is reminiscent of the essential gain problem that one may find more familiar with. The essential gain of a matrix $C$ is defined as
\[
\overline\sigma_{\mathrm{ess}}(C)=\inf_{D\in\mathcal{D}} \overline\sigma(D^{-1}CD),
\]
which has proven useful in various applications. It has been studied in \cite{Safonov1982} with input from \cite{Bauer}.

It is known that the essential gain problem can be solved numerically but does not have an analytic solution in general. In the case of a nonnegative matrix $C$, the essential gain has an analytic expression $\overline\sigma_{\mathrm{ess}}(C)=\rho(C)$ and the optimal scaling matrix $D$ can be obtained from the Perron eigenvectors of $C$ \cite{StoerWitzgall}, where $\rho(C)$ denotes the spectral radius of $C$. It is a similar situation for the essential phase problem. In general the problem can be solved numerically. For some special classes of matrices, analytic or semi-analytic solutions can be obtained. Of particular interest are the essential phases of Laplacian matrices.

Before proceeding, we introduce some preliminaries on Laplacian matrices of graphs.
A graph, denoted by $\mathbb{G}\!=\!(\mathcal{V}, \mathcal{E})$, consists of a set of nodes $\mathcal{V}$ and a set of edges $\mathcal{E}$. We use $(i,j)$ to represent the edge directed from node $i$ to node $j$.
A path from node $i_1$ to node $i_k$ is a sequence of edges $(i_1, i_2), (i_2, i_3),\dots,(i_{k-1}, i_k)$ with $(i_j, i_{j+1})\in \mathcal{E}$ for $j\in \{1,\dots,k-1\}$. A node is called a root if it has paths to all the other nodes in the graph. A graph $\mathbb{G}$ is said to have a spanning tree if it has a root. Furthermore, $\mathbb{G}$ is said to be strongly connected if every node is a root. A graph is undirected if $(i,j)\in \mathcal{E}$ implies $(j,i)\in\mathcal{E}$.

A weighted graph is a graph with each edge associated with a weight. The edge weights are assumed to be positive. Denote by $a_{ji}$ the weight of edge $(i,j)$, where $a_{ji}$ is understood to be zero when there is no edge from node $i$ to $j$. The indegree and outdegree of node $i$ are given by $d_{\mathrm{in}}(i)=\sum_{j=1}^n a_{ij}$ and $d_{\mathrm{out}}(i)=\sum_{j=1}^n a_{ji}$ respectively. A graph is said to be weight-balanced if $d_{\mathrm{in}}(i)=d_{\mathrm{out}}(i)$ for all $i\in\mathcal{V}$. For a weighted graph, its Laplacian matrix $L=[l_{ij}]$ is defined as
\begin{align*}
l_{ij}=\begin{cases}
-a_{ij}, & i\neq j,\\
\sum_{j=1,j\neq i}^n a_{ij}, &i=j.
\end{cases}
\end{align*}
The Laplacian matrix of a strongly connected graph is irreducible, i.e., not similar via a permutation to a block upper triangular matrix. All of the eigenvalues of a Laplacian matrix $L$ lie in the closed right half plane. Also, $L$ has a zero eigenvalue with a corresponding eigenvector being $\mathbf{1}_n$. Furthermore, $0$ is a simple eigenvalue of $L$ if and only if $\mathbb{G}$ has a spanning tree. See \cite{laplacian-survey} for a survey on Laplacian matrices.

We first consider the essential phases of the Laplacian of a strongly connected graph. In general, the Laplacian matrix $L$ is not semi-sectorial. We have the following result.

\begin{lemma}\label{bal}
Let $\mathbb{G}$ be a strongly connected directed graph and $L$ be its Laplacian matrix. The following statements are equivalent:
\begin{enumerate}
\item $L$ is quasi-sectorial.
\item $L$ is semi-sectorial.
\item $\mathbb{G}$ is weight-balanced.
\end{enumerate}
\end{lemma}

To prove Lemma \ref{bal}, we review a lemma on the numerical range of a nonnegative matrix.
The numerical radius of $C\in\mathbb{C}^{n\times n}$ is given by
$w(C)\!=\!\max\{|z|: z\in W(C)\}$.

\begin{lemma}[\cite{CKLi2002}]
\label{lemma:Li-sharp}
Let $C$ be an irreducible nonnegative matrix. Then the following statements are equivalent:
\begin{enumerate}
\item $w(C)$ is a sharp point of $W(C)$;
\item $w(C)=\rho(C)$;
\item $C$ has a common left and right Perron eigenvector.
\end{enumerate}
\end{lemma}

{\em Proof of Lemma~\ref{bal}:} The implication 1$\implies$2 follows directly from the definition of quasi-sectorial and semi-sectorial matrices.

Next we show 2$\implies$3. Since $L$ is a singular semi-sectorial matrix, it has a decomposition of the form (\ref{gsd}). This implies that $L$ has a common left and right eigenvector corresponding to the zero eigenvalue. Therefore, $L\mathbf{1}=L'\mathbf{1}=0$ which implies that $\mathbb{G}$ is weight-balanced.

Finally we show 3$\implies$1. We can express $L$ as $L=\rho(B)I-B$ for a nonnegative matrix $B$. Since $\mathbb{G}$ is weight-balanced, $L$ has $\mathbf{1}$ being a common left and right eigenvector corresponding to eigenvalue $0$. Therefore, $B$ has a common left and right eigenvector corresponding to eigenvalue $\rho(B)$. Then, by Lemma \ref{lemma:Li-sharp}, $\rho(B)=w(B)$ is a sharp point of $W(B)$. It follows that $0$ is a sharp point of $W(L)$, implying that $L$ is quasi-sectorial. 

\vspace{8pt}

Lemma \ref{bal} is the key in finding the essential phases of a Laplacian matrix. For a strongly connected graph, $L$ has a positive left eigenvector $v$ corresponding to the zero eigenvalue, i.e., $v'L=0$. Let $V=\mathrm{diag}\{v\}$ and $D_0=V^{-1/2}$.

\begin{lemma}\label{essphaseLap}
Let $\mathbb{G}$ be a strongly connected directed graph and $L$ be its Laplacian matrix. Then
\begin{align*}
\overline\phi_{\mathrm{ess}}(L)&=\overline\phi(D_0^{-1}LD_0)=\overline\phi(VL),\\
\underline\phi_{\mathrm{ess}}(L)&=\underline\phi(D_0^{-1}LD_0)=\underline\phi(VL).
\end{align*}
\end{lemma}

\begin{proof}
Observe that $VL$ is a Laplacian matrix with $\mathbf{1}$ being a common left and right eigenvector corresponding to eigenvalue $0$. This means that $VL$ is the Laplacian matrix of a weight-balanced graph.
By Lemma \ref{bal}, $VL$ is quasi-sectorial. Hence $D_0^{-1}LD_0$ is quasi-sectorial as it is congruent to $VL$. Furthermore, Lemma \ref{bal} implies that if a $D\in\mathcal{D}$ makes $D^{-1}LD$ semi-sectorial, it in fact makes it quasi-sectorial. In addition, such diagonal scaling matrix $D$ is unique up to positive number multiplication. This completes the proof. 
\end{proof}

A piece of information hidden in the above proof is that $D_0^{-1}LD_0$ is quasi-sectorial and thus $\overline{\phi}_{\mathrm{ess}}(L)<\frac{\pi}{2}$.
It then follows from \cite[Lemma 2.3]{WCKQ2020} that
\[
\max_i \{\angle \lambda_i(L)\} \leq  \overline\phi_{\mathrm{ess}}(L)<\pi/2.
\]
Since $L$ is real, there holds $\underline \phi_{\mathrm{ess}}(L)\!=\!-\overline\phi_{\mathrm{ess}}(L)$. For
this reason, hereinafter we use $\phi_{\mathrm{ess}}(L)$ to represent $\overline\phi_{\mathrm{ess}}(L)$ for notational simplicity.
In the case of an undirected graph, $L$ is symmetric and hence $\phi_{\mathrm{ess}}(L) = 0$. This suggests the use of $\phi_{\mathrm{ess}}(L)$ as a measure of ``directedness'' of a graph.

We proceed to consider the case where the graph is not strongly connected but has a spanning tree. In this case, one can decompose the graph into multiple strongly connected components. Suppose the graph has $n_1$ roots and $m$ strongly connected components. Without loss of generality, one can relabel the nodes to form $m$ groups
\begin{align}
\!\{1,\dots,n_1\},\{n_1\!+\!1,\dots,n_2\},\dots,\{n_{m-1}\!+\!1,\dots,n\}\label{strongcd}
\end{align}
so that the nodes in each group correspond to a strongly connected component and the first component contains all the roots. The Laplacian $L$ can be written accordingly in the Frobenius normal form \cite{BrualdiRyser}
\begin{equation}
\begin{split}
L=\begin{bmatrix}
L_{11} & 0& \cdots &0\\
L_{21} & L_{22} & \cdots &0\\
\vdots &\vdots &\ddots &\vdots\\
L_{m1} & L_{m2} &\cdots & L_{mm} \label{eq: L_decomp}
\end{bmatrix},
\end{split}
\end{equation}
where $L_{11}$ is the Laplacian of the subgraph induced by all the roots and $L_{kk}$, $k=2,\dots,m$ are nonsingular irreducible M-matrices\footnote{A matrix $C\in\mathbb{R}^{n\times n}$ is said to be an M-matrix if it can be written as $C=sI-A$, where $A$ is nonnegative and $s\geq\rho(A)$.} and are diagonally dominant.
Moreover, $L$ has a nonnegative left eigenvector $v=\begin{bmatrix}v_1'&0\end{bmatrix}'$ corresponding to eigenvalue $0$, where $v_1$ is a positive left eigenvector of $L_{11}$ corresponding to eigenvalue $0$. Since $v$ is not positive, Lemma \ref{essphaseLap} fails to hold in this case.

Nevertheless, one often needs to find the essential phase of each $L_{kk}$ on the diagonal. Clearly, $\phi_{\mathrm{ess}}(L_{11})$ can be determined as in Lemma \ref{essphaseLap} for $L_{11}$ is the Laplacian associated to the first strongly connected component.
The following lemma shows that $\phi_{\ess}(L_{kk}),k\!=\!2,\dots,m$ exist and are bounded by $\phi_{\mathrm{ess}}(\tilde{L}_k),k\!=\!2,\dots,m$ respectively, where $\phi_{\ess}(L_{kk})$ represents $\overline{\phi}_{\ess}(L_{kk})$ and $\tilde{L}_k$ is the Laplacian matrix of the $k$th strongly connected component of the graph, which can be obtained by reducing the diagonal elements of $L_{kk}$ to zero all row sums.

\begin{lemma}\label{epLii}
Let $\mathbb{G}$ be a directed graph with a spanning tree and $L$ be its Laplacian matrix in the form of (\ref{eq: L_decomp}). Then $\phi_{\mathrm{ess}}(L_{kk})\leq \phi_{\mathrm{ess}}(\tilde{L}_k),k=2,\dots,m$.
\end{lemma}

\begin{proof}
Note that $L_{kk}=\tilde{L}_k+Z_k,k=2,\dots,m$, where $Z_k$ is a diagonal matrix with nonnegative diagonal elements and $Z_k\neq 0$. Let $v_k$ be a positive left eigenvector of $\tilde{L}_k$ corresponding to eigenvalue $0$ and $D_k=\mathrm{diag}\{v_k\}^{-\frac{1}{2}}$. Then $D_k^{-1}L_{kk}D_k\!=\!D_k^{-1}\tilde{L}_k D_k+Z_k$.
By Lemma \ref{essphaseLap}, we know that $D_k^{-1}\tilde{L}_k D_k$ is quasi-sectorial and
\begin{align*}
    \phi_{\mathrm{ess}}(\tilde{L}_k)=\overline{\phi}(D_k^{-1}\tilde{L}_k D_k)<\pi/2.
\end{align*}
Since $Z_k\geq 0$, it follows that
\begin{align*}
\overline{\phi}(D_k^{-1}L_{kk}D_k)=\overline{\phi}(D_k^{-1}\tilde{L}_k D_k+Z_k)\leq \phi_{\mathrm{ess}}(\tilde{L}_k).
\end{align*}
By definition, we have
$\phi_{\mathrm{ess}}(L_{kk})\!\leq\! \overline{\phi}(D_k^{-1}L_{kk}D_k)$ and thus $\phi_{\mathrm{ess}}(L_{kk})\!\leq\! \phi_{\mathrm{ess}}(\tilde{L}_k)$.
This completes the proof. 
\end{proof}

Lemma \ref{epLii} provides an upper bound of the essential phase for a nonsingular diagonally-dominant irreducible M-matrix. However, a general M-matrix maynot be diagonally-dominant. For this case, Lemma \ref{epLii} may not hold.
\begin{example}
Consider an M-matrix
\begin{align*}
M= 1.0691 I -\begin{bmatrix}
0.5338 & 0.3381  & 0.0103\\
0.1092  &  0.2940  &  0.0484\\
0.8258  &  0.7463  &  0.6679
\end{bmatrix}.
\end{align*}
The essential phase of the associated Laplacian is $0.1403$ while $\phi_{\ess}(M)= 0.1662$.
\end{example}

The computation of the essential phase of an M-matrix will be studied in the next section.

\section{Computation of Essential Phases of M-Matrices}
In this section, we will study the essential phase of a general M-matrix. We first show that the essential phase of a general nonsingular irreducible M-matrix exists and provide an upper bound. The matrix $M$ can be written into the form $M=sI-A$, where $A$ is a irreducible nonnegative matrix and $s\geq\rho(A)$. According to Perron-Frobenius Theorem, the matrix $M$ has positive left and right eigenvectors $x$ and $y$ respectively corresponding to the eigenvalue $s-\rho(A)$. Let $D_0\!=\!\mathrm{diag}(\sqrt{x_1/y_1},\dots,\sqrt{x_n/y_n})$. We have the following lemma, whose proof follows directly from Lemma \ref{lemma:Li-sharp} and thus is omitted for brevity.

\begin{lemma}\label{upper}
Let $M$ be a nonsingular irreducible M-matrix. Then $\phi_{\mathrm{ess}}(M)\!\leq\!\overline\phi(D_0^{-1}MD_0)$.
\end{lemma}

It is known that $\max_i\angle\lambda_i(M)$ serves as a lower bound of $\phi_{\mathrm{ess}}(M)$. With the upper bound and the lower bound, in the sequel, we aim to propose an algorithm to numerically compute $\phi_{\ess}(M)$.

According to the fact that phases are preserved under congruence transformation and $D$ is positive diagonal, we have
\[
\overline\phi(D^{-1}MD)=\overline\phi(D^{-T}D^{-1}M)=\overline\phi(D^{-2}M).
\]
Note that $D^{-2}$ is positive diagonal, hence
\begin{align*}
\phi_{\mathrm{ess}}(M)\!=\!\inf_{D\in\mathcal{D}} \overline\phi(D^{-1}MD)=\!\inf_{D\in\mathcal{D}} \overline\phi(DM),
\end{align*}
where $\mathcal{D}$ is a set of positive definite diagonal matrices. The computation of $\phi_{\mathrm{ess}}(M)$ can be written as an optimization problem
\begin{align}\label{opt:original}
\inf_{D\in\mathcal{D}}\quad&\overline\phi(DM).
\end{align}
The epigraph form of the problem (\ref{opt:original}) is
\begin{align}\label{opt:epi}
\inf\{\alpha: \overline\phi(DM)\leq\alpha,D\in\mathcal{D}\}.
\end{align}
Since $\overline\phi(DM)\leq\alpha$ is equivalent to $e^{j(\frac{\pi}{2}-\alpha)}DM$ is accretive, i.e., $e^{j(\frac{\pi}{2}-\alpha)}DM+e^{-j(\frac{\pi}{2}-\alpha)}M^TD\geq 0$, the problem (\ref{opt:epi}) can be further written as
\begin{align}\label{opt:cos}
\inf\{\alpha:(\sin\alpha+j\cos\alpha)DM+(\sin-j\cos\alpha)M^TD\geq 0,D\in\mathcal{D}\}.
\end{align}
It is shown in the last section that the optimal value of problem (\ref{opt:cos}) $\alpha^*$ lies in $[0,\frac{\pi}{2}]$. Since whether $\alpha^*=0$ can be easily verified by checking whether there exists $D\in\mathcal{D}$ such that $DM>0$, hereinafter we assume that $\alpha^*>0$. It follows that $\sin\alpha>0$. Therefore, the problem is translated to
\begin{align}
\inf\{\alpha:(1+j\cot\alpha)DM+(1-j\cot\alpha)M^TD\geq 0,D\in\mathcal{D}\}.
\end{align}
This is an optimization problem over bilinear matrix inequality constraint, which might be NP-hard \cite{vanantwerp2000tutorial}. However, as the upper bound and lower bound of the objective function can be obtained, the bisection algorithm can be used to solve the problem. Here we choose the initial lower bound to be $0$ owing to the fact that $\max_i\angle\lambda_i(M)$ is not easy to be obtained. The detailed algorithm is given in Algorithm 1.
\begin{algorithm}
	\caption{An algorithm for computing the essential phase}\label{alg:cap}
	\begin{algorithmic}[1]
		\Require Matrix $M$, a lower bound of $\underline\alpha=0$, an upper bound $\bar\alpha=\overline\phi(D_0^{-1}MD_0)$, absolute error $e$, i.e., the desired degree of accuracy.
		\Ensure The optimal value $\alpha^*$.
	    \If {there exists $D$ such that $DM\geq 0$}
		\State $\alpha^*\gets 0$
		\Else
		\State $\beta=(\bar\alpha+\underline\alpha)/2$
		\While {$\bar\alpha-\underline\alpha\geq e$}
		\If {there exists $D$ such that
			\[(1+j\cot\beta)DM+(1-j\cot\beta)M^TD\geq 0\]}
		\State $\bar\alpha\gets\beta$
		\Else
		\State $\underline\alpha\gets\beta$
		\EndIf
		\State $\beta=(\bar\alpha+\underline\alpha)/2$
		\EndWhile
		\State $\alpha^*\gets \overline\alpha$
		\EndIf
	\end{algorithmic}
\end{algorithm}

We also want to point out that the algorithm for the computation of essential phases can be easily generalized to arbitrary square matrices by giving an initial guess of upper and lower bounds.
\begin{example}
Consider an M-matrix
\begin{align*}
M=3I-
\begin{bmatrix}
 0.8147  &  0.6324  &   0.9575  &   0.9572\\
    0.9058 &    0.0975  &   0.9649   &  0.4854\\
    0.1270 &    0.2785 &    0.1576 &    0.8003\\
    0.9134 &    0.5469  &   0.9706 &    0.1419
\end{bmatrix},
\end{align*}
which has an eigenvalue at $0.5978$ with an associated right eigenvector
\[
\begin{bmatrix}0.6621& 0.4819 &0.2766& 0.5029\end{bmatrix}'
\]and left eigenvector
\[
\begin{bmatrix}0.5308 &0.3371& 0.5902 &0.5062\end{bmatrix}'.
\]
Then the upper bound is given by $\overline\phi(D_0^{-1}MD_0)=0.1053.$ Choose absolute error $e=10^{-5}$. Applying the bisection algorithm yields $\phi_{\mathrm{ess}}(M)=0.0973$.
\end{example}

\section*{Acknowledgement}
This work was partially supported by the Research Grants Council of Hong Kong Special Administrative Region, China, under the General Research Fund 16201120 and the National Natural Science Foundation of China under grants 62073003, 72131001.

\bibliographystyle{elsarticle-num}
\bibliography{ctphase,dynamic_v2}

\end{document}